\documentclass[reqno, 11pt]{amsart}
\usepackage{srcltx}
\usepackage[mathscr]{eucal}
\usepackage{amscd}
\usepackage{color}
\usepackage{amsfonts}
\usepackage{cleveref}
\usepackage{amsmath,amssymb,amsfonts,amsthm,bbm,mathrsfs,verbatim}
\usepackage{latexsym}
\usepackage{enumerate}
\usepackage[misc]{ifsym}
\usepackage{fancyhdr}
\usepackage{xr-hyper}

\usepackage{graphics}
\usepackage{epsfig}
\usepackage{xcolor}
\usepackage{graphicx}
\usepackage{epstopdf}
 \newtheorem{thm}{Theorem}[section]
 
 \newtheorem{lem}[thm]{Lemma}

 \newtheorem{rem}[thm]{Remark}
 
 \numberwithin{equation}{section}

\def\R{{\Bbb R}}

\def\la{{\lambda}}

\def\R{{\mathbb R}}

\def\Z{{\Bbb Z}}
\def\T{{\Bbb T}}
\def\S{{\Bbb S}}

\def\ve{{\varepsilon}}

\def\ve{{\varepsilon}}

\def\Om{{\Omega}}
\def\bge{\begin{eqnarray}}
\def\bgee{\begin{eqnarray*}}
\def\ege{\end{eqnarray}}
\def\egee{\end{eqnarray*}}

\newcommand{\ee}{{\mathrm{e}}}

\def\be{\begin{equation}}
\def\ee{\end{equation}}
\def\bse{\begin{subequations}}
\def\ese{\end{subequations}}
\def\bge{\begin{eqnarray}}
\def\bgee{\begin{eqnarray*}}
\def\ege{\end{eqnarray}}
\def\egee{\end{eqnarray*}}

\date{\today}
\dedicatory{Dedicated to Professor Nicholas Alikakos on the occasion of his retirement}
\date{\today}
\begin{document}
\title[]
{Gradient inequality and convergence of the normalized Ricci flow}

\author{Nikos I. Kavallaris}
\address{Department of Mathematical and Physical Sciences, University of Chester, Thornton Science Park,
Pool Lane, Ince, Chester  CH2 4NU, UK}
\email{n.kavallaris@chester.ac.uk}

\author{Takashi Suzuki}
\address{Center for Mathematical Modeling and Data Science, Osaka University,
}

\email{suzuki@sigmath.es.osaka-u.ac.jp}

%35J60; Secondary 74H35, 74G55, 74K15


\date{\today}

%\linenumbers

%\date{}
\pagenumbering{arabic}

%_________________________________________________________________

\begin{abstract}
We study the problem  of convergence of the normalized Ricci flow  evolving  on a compact manifold $\Omega$ without boundary. In \cite{KS10, KS15} we derived, via PDE techniques, global-in-time existence of the classical solution and pre-compactness of the orbit. In this work we show its convergence to steady-states, using a gradient inequality of {\L}ojasiewicz type. We have thus an alternative proof of \cite{ha}, but  for general manifold $\Omega$ and not only for unit sphere. As a byproduct of that approach we also derive the rate of convergence according to  this  steady-sate being  either degenerate or non-degenerate as a critical point of a related energy functional.
\end{abstract}

\maketitle
\vspace{0.5in}

\section{Introduction}\label{sec1}

In the current work we revisit the following problem of logarithmic diffusion
\begin{eqnarray}
& & u_t=\Delta \log u+u-\frac{1}{|\Omega|}\int_\Omega u\, dx, \qquad \;\,x\in \Omega,\;\;t>0\label{ieq1}\\
& & u(x,0)=u_0(x)>0, \qquad \qquad \qquad \qquad \quad x\in \Omega \label{ieq2}
\end{eqnarray}
with
\begin{equation}\label{consv}
\int_{\Omega} u_0(x) \ dx=\lambda, 
\end{equation}
where $\lambda>0$ is a constant and $\Omega$ is a compact Riemannian surface without boundary. When $\lambda=8\pi$ and  $\Omega$ equals to the unit sphere $\S^2=\{x\in \R^2: ||x||<1\}$, problem  \eqref{ieq1} - \eqref{consv} describes an evolution of the metric $g=g(t)$ on $\Omega$, that is, the normalized Ricci flow introduced by \cite{ha} as 
\[ \frac{\partial g}{\partial t}=(r-R)g, \] 
where $R$ is the scalar curvature, $r$ is the volume mean 
\[ r=\frac{\int_\Omega R d\mu}{\int_\Omega d\mu}, \] 
and $\mu=\mu(t)$ is the area element, cf.  \cite{bsy, KS18, S15, S20}. 

The standard parabolic theory, cf.  \cite{lsu68},  assures that for smooth initial data $u_0(x)>0$ problem  \eqref{ieq1} - \eqref{ieq2} has  a unique classical solution local in time, denoted by $u(x,t)>0$ in $\Omega\times (0,T)$ with $T=T_{max}>0;$. there also  holds that  
\[ 
\int_{\Omega} u(x,t)\,dx=\int_{\Omega} u_0(x)\,dx\equiv \lambda>0 
\] 
by (\ref{ieq2}). The main result in  \cite{ha} reads as:  if 
\begin{equation} 
\lambda=8\pi, \quad \Omega=\S^2,  
 \label{nrf}
\end{equation} 
then there holds that $T=+\infty$ and   
\begin{equation} \label{ikk2}
u(\cdot,t)\to u^*, \ t\to\infty\quad \mbox{in $C^{\infty}$ topology},
\end{equation} 
where $u^*=u^*(x)>0$  is a stationary solution to  \eqref{ieq1} - \eqref{consv}: 
\begin{equation} 
-\Delta \log u^\ast=u^\ast-\frac{1}{|\Omega|}\int_\Omega u^\ast\, dx, \ \ x\in \Omega, \quad \int_\Omega u^\ast \ dx =\lambda. 
 \label{stationary}
\end{equation} 
Remarkably, Hamilton's approach in \cite{ha}, uses the geometric structure of  problem \eqref{ieq1}-\eqref{consv}  valids only for the special case (\ref{nrf}), which results in the control of several key geometric quantities related to the evolution of the metric $g=g(t).$ 

On the other hand, via a PDE approach, which works to any $0<\la\leq 8\pi$ and  any compact Riemannian surface $\Omega$ without boundary,  we could derive convergence \eqref{ikk2}, cf. \cite{KS10, KS15}. Specifically,  in spite of the lack of geometric structure of this general case, the following holds.  
\begin{thm}[\cite{KS10, KS15}] 
Assume that $\Omega$ is a compact Riemannian surface without boundary and $0<\lambda\leq 8\pi$. Then the solution $u=u(x,t)$ to  \eqref{ieq1} - \eqref{consv} exists global in time and satisfies the uniform estimates 
\begin{equation} 
\sup_{t\geq 0}\left\{||u(\cdot, t)||_{\infty}+||u^{-1}(\cdot, t)||_{\infty}\right\} <\infty. 
 \label{compact}
\end{equation} 
 \label{thm1} 
\end{thm}

We note that Theorem \ref{thm1} reproduces the convergence (\ref{ikk2}) for the case  (\ref{nrf}) with the aid of dynamical and elliptic theories. 

The first step to confirm this fact is to notice that system (\ref{ieq1})-(\ref{consv}) is provided with a Lyapunov function and $L^1$ conservation. In fact, problem  \eqref{ieq1} is written as the parabolic-elliptic system 
\begin{eqnarray}
& & u_t=\Delta (\log u-v), \label{eqn:b-1} \\
& & -\Delta v=u-\frac{1}{\left\vert \Omega\right\vert}\int_\Omega u \ dx , \ \ \int_\Omega v\ dx=0, \qquad x\in \Omega, \ \ t>0. 
 \label{eqn:b-2}
\end{eqnarray}
Notably  (\ref{eqn:b-1})-(\ref{eqn:b-2}) takes the form of model (B) equation studied by \cite{KS18, S15}, that is
\[ 
u_t=\Delta \delta\mathcal{F}(u),\quad u(x,0)=u_0(x),\quad x\in \Omega,\;t>0,\label{ele1}
\] 
where $\delta\mathcal{F}$ stands for the first variation of the functional
\bgee
\mathcal{F}(u)=\int_{\Om} u\left(\log u-1\right)\,dx+\frac{1}{2}\int_{\Om} u \cdot \Delta^{-1} u\,dx
\egee
whilst by 
\[ v=-\Delta^{-1}u \] 
we mean that $v$ satisfies \eqref{eqn:b-2} for given $u$. As a consequence it holds that 
\begin{equation} 
\frac{d}{dt}\int_\Omega u=0, \quad \frac{d}{dt}{\mathcal F}(u)=\langle u_t, \delta {\mathcal F}(u)\rangle =-\Vert \nabla \delta {\mathcal F}(u)\Vert_2^2\leq 0 
 \label{modelb}
\end{equation} 
for the solution $u=u(\cdot,t)$. 

By this variational structure of (\ref{modelb}), the steady-state $u^*$ of \eqref{eqn:b-1}-\eqref{eqn:b-2} is defined by 
\begin{equation} 
 \delta\mathcal{F}(u^*)= \log u^*+\Delta^{-1}u^*=\mbox{constant}\;,\int_{\Om} u^*\,dx=\la,\; u^*=u^*(x)>0.  
  \label{112}
\end{equation} 
For more details regarding the above formulation of the steady-state as a model (B) equation see \cite{S15}. 

Since Theorem \ref{thm1} guarantees the pre-comapcness  in $C^\infty$ topology of the orbit $\{ u(\cdot,t)\}$ of  the global-in-time solution to (\ref{ieq1})-(\ref{consv}), the following fact arises by the theory of dynamical systems, that is, the LaSalle principle  \cite{henry, S15}, where 
\begin{eqnarray*} 
 \omega(u_0)=\{ u_\ast=u_\ast(x)>0 \mid \mbox{there exists $t_k\rightarrow\infty$ such that}  
 \quad \mbox{$u(\cdot,t_k)\rightarrow u_\ast$ in $C^\infty$ topology}\} 
\end{eqnarray*} 
denotes the $\omega$-limit set. 

\begin{thm}\label{thm11}
Under the assumption of Theorem \ref{thm1},  the $\omega$-limit set $\omega(u_0)$ is non-empty, connected, and compact, contained in the set of stationary solutions denoted by 
\[ F_\lambda=\{ u^\ast=u^\ast(x)>0 \mid \mbox{classical solution to (\ref{112})} \}. \] 
\end{thm} 

\begin{rem} 
To show  the equivalence of (\ref{stationary}) and (\ref{112}), first, put $v^\ast=-\Delta^{-1}u^\ast$ in (\ref{112}). 
Then (\ref{112}) implies  
\begin{eqnarray} 
& & -\Delta v^\ast=u^\ast-\frac{1}{\vert \Omega\vert}\int_\Omega u^\ast \ dx, \quad \int_\Omega v^\ast=0 \\ 
& & \log u^\ast=v^\ast+\mbox{constant}, \quad \int_\Omega u^\ast=\lambda 
 \label{trans2}
\end{eqnarray} 
and hence 
\begin{eqnarray}
& & -\Delta v^\ast=\lambda \left(\frac{e^{v^\ast}}{\int_{\Omega}e^{v^\ast}\,dx}-\frac{1}{|\Omega|}\right), \quad \int_{\Omega}v^\ast\,dx=0  \label{113} \\  
& & u^\ast=\frac{\lambda e^{v^\ast}}{\int_\Omega e^{v^\ast} dx}, \label{0114} 
\end{eqnarray}  
which implies (\ref{stationary}). If $u^\ast=u^\ast(x)>0$ solves (\ref{stationary}),  then (\ref{113})  arises  for 
\begin{equation} 
v^\ast=w^\ast-\frac{1}{\vert \Omega\vert}\int_\Omega w^\ast \ dx, \quad w^\ast=\log u^\ast, 
 \label{transform}
\end{equation} 
and hence (\ref{112}) holds true. Thus (\ref{stationary}) is equivalent to (\ref{112}). 
\end{rem} 

Theorem \ref{thm11} implies (\ref{ikk2}) if $F_\lambda$ is discrete, particularly, a singleton:
\begin{equation}  
F_\lambda=\left\{ \lambda/\vert \Omega\vert\right\}. 
 \label{singleton}
\end{equation} 
Under the transformation (\ref{trans2}), furthermore, property (\ref{singleton}) is equivalent to 
\begin{equation} 
E_\lambda=\{0\}, 
 \label{114}
\end{equation}  
where 
\[ E_\lambda=\{ v^\ast \mid \mbox{solution to (\ref{113})} \}. \] 
If (\ref{114}) holds for $0<\lambda\leq 8\pi$, then arises (\ref{ikk2}) with 
\begin{equation} 
u^\ast=\frac{\lambda}{\vert \Omega\vert} 
 \label{8pi} 
\end{equation} 
in (\ref{ieq1})-(\ref{consv}) by Theorem \ref{thm11}.   

We are ready to begin the second step of deriving (\ref{ikk2}) for (\ref{nrf}) by Theorem \ref{thm11}, that is, the confirmation of  (\ref{114}) for (\ref{nrf}). First of all, we note that this fact follows from a geometric property, that is, the classification of the closed surface with constant Gaussian curvature. Direct proof in the context of the elliptic theory, however, is also available \cite{ck, cl, l}. The analytic proof of  (\ref{ikk2}) for (\ref{nrf}) is thus complete.  

\begin{rem} 
Another case when (\ref{114}) is valid is for 
\begin{equation} 
\lambda=8\pi, \quad \Omega=\T\equiv \R^2/a\Z\times b\Z, \ \frac{b}{a}\geq \frac{\pi}{4} 
 \label{torus}
\end{equation} 
by the elliptic theory \cite{ll}. Hence,  (\ref{ikk2}) arises with (\ref{8pi}) and for (\ref{torus}).   
\end{rem} 
The convergence (\ref{ikk2}), however, is valid even if $E_\lambda$ forms a continuum. Furthermore, we can even determine the rate of convergence and thus Theorem \ref{thm11} is  improved as follows. 
\begin{thm}\label{thm2}
Under the assumption of Theorem \ref{thm1} there is a solution $u^*=u^\ast(x)>0$ to  (\ref{stationary}) so that  (\ref{ikk2}) is valid. The rate of this convergence  is at least algebraic. 
\end{thm}

\begin{thm}\label{thm14}
If $u^*$ is non-degenerate in the previous theorem, the rate of convergence in (\ref{ikk2}) is exponential. 
\end{thm}

To define the non-degeneracy of the steady-state  $u^\ast$, we use the fact that $v^\ast$ defined by (\ref{transform}) is a solution to (\ref{113}), which is the Euler-Langange equation of the energy functional 
\begin{equation} 
J_{\la}(v)=\frac{1}{2} ||\nabla v||_2^2-\la \log \int_{\Om} e^v\,dx
 \label{functionalj}
\end{equation} 
defined for $v\in V_0$, where
\begin{equation}  
V_0=\{v\in H^1(\Om) \mid \int_{\Om} v \ dx=0\}. 
 \label{functionspace}
\end{equation} 
Thus we say that $u^\ast=u^\ast(x)>0$ is non-degenerate  in Theorem \ref{thm14}, if $v^\ast\in V_0$ defined by (\ref{transform}) is non-degenerate as a critical point of $J_\lambda$ on $V_0$. Later in Lemma \ref{lem0}, we show that this non-degeneracy of $u^\ast=u^\ast(x)>0$ means that 
\begin{equation} \label{ikk39}
\psi \in H^2(\Omega), \ -\Delta \psi=u^*\psi \ \mbox{in $\Om$}, \ \int_{\Om} \psi u^*\,dx=0 \quad \Rightarrow \quad \psi=0. 
\end{equation} 

\begin{rem} 
We can regard (\ref{113}) as a nonlinear eigenvalue problem of finding $(\lambda, v^\ast)$ similtaneously. Then, if $\Omega=\S^2$, non-trivial solutions bifurcate at $\lambda=8\pi$ from the branch of trivial solutions $\{(\lambda, v^\ast)\mid \lambda\in \R, \ v^\ast=0\}$. Hence we cannot apply Theorem \ref{thm14} for this case, but still have the rate at least of algebraic order in (\ref{ikk2}) with (\ref{8pi}) for $\lambda=8\pi$. In the case of $\Omega=\R^2/a\Z\times b\Z$ with $\frac{b}{a}\geq \frac{\pi}{4}$, on the other hand, $\lambda=8\pi$ is not a bifurcation point of the non-trivial solution, and $v^\ast=0$ is still non-degnerate at this value of $\lambda=8\pi$. Hence in the case of  a torus given by (\ref{torus}), there holds (\ref{ikk2}), with (\ref{8pi}) for $\lambda=8\pi$, in the exponential rate. 
\end{rem} 

Theorem \ref{thm1}, without geometric structure for  problem (\ref{ieq1})-(\ref{consv}), is proven as follows. First, the range  $0<\lambda<8\pi$ of this problem is sub-critical in accordance with the Trudinger-Moser-Fontana inequality \cite{f} 
\begin{equation}\label{ftm}
v\in V_0, \ \Vert \nabla v\Vert_2\leq 1 \quad \Rightarrow \quad \int_{\Om} e^{4\pi v^2}\,dx\leq C 
\end{equation} 
which entails 
\[ 
\inf\{J_{8\pi}(v)\Big| v\in V_0\}>-\infty 
\] 
as in \cite{S15}. Hence Moser's iteration ensures $T=+\infty$ and (\ref{compact}) for $0<\lambda<8\pi$, under 
\[ \Vert u(\cdot,t)\Vert_1=\lambda, \quad \mathcal{F}(u(\cdot,t))\leq \mathcal{F}(u_0) \] 
derived from (\ref{modelb}), cf. \cite{KS10}.  On the other hand, Benilan-Crandall's inequality
\[ 
\frac{u_t(x,t)}{u(x,t)}\leq \frac{e^t}{e^t-1}
\] 
is used to confirm $T=+\infty$ for $\lambda=8\pi$ (\cite{KS10}). To derive (\ref{compact}) for $\la=8 \pi$,  we finally appeal to a concentration compactness argument, cf.  \cite{KS15}. 

\begin{rem}
An immediate consequence of (\ref{ftm}) is that any $K>0$ admits $C(K)>0$ such that 
\[ v\in V=H^1(\Omega), \ \Vert v\Vert_V\leq K \quad \Rightarrow \quad \Vert e^{\vert v\vert}\Vert_1\leq C(K), \] 
where $\Vert v\Vert_V=(\Vert v\Vert_2^2+\Vert \nabla v\Vert_2^2)^{1/2}$. 
\end{rem} 

The main aim of the current workfor is to provide an analytic  proof of Theorems \ref{thm2}-\ref{thm14},  by using  a gradient inequality  which takes the following classical form in the finite dimensional case. 

\begin{lem}[\cite{L63}]
Let $E=E(x):\R^n\rightarrow \R$ be real-analytic at $x=0$, satisfying $E(0)=0$ and $\delta E(0)=0$.  Then there is $0<\theta\leq \frac{1}{2}$ such that 
\[ \vert E(x)\vert^{1-\theta}\leq C\vert \delta E(x)\vert, \quad \vert x\vert \ll 1. \] 
 \label{cgi}
\end{lem} 

We  thus provide an alternative proof of  Hamilton's convergence result in \cite{ha}  which is entirely  based on the parabolic theory; that is no use of  any geometric or  elliptic structures of \eqref{ieq1}-\eqref{consv} is made. Furthermore, our proof assures (\ref{ikk2}) for any $0<\lambda\leq 8\pi$ and a general comapct manifold $\Omega$ without boundary and it also shows that the convergence rate is at least algebraic and exponential, provided that $u^\ast=u^\ast(x)>0$ is degenerate and non-degenerate  steady-state respectively.

This paper is composed of five sections. Section \ref{sec2} is devoted to the key concept of critical manifold developed in \cite{Ch03, Ch06}. Then, Theorem \ref{thm2} is proven in Section \ref{sec3}, employing the method presented  in  \cite{HJ01, S83}. Section \ref{sec4} is devoted to the study of  non-degenerate steady-state solution, and finally Theorem \ref{thm14} is proven in Section \ref{sec5}.

\vspace{2mm}

{\bf Notations.} In the sequel $||\cdot ||_p$ denotes the $L^p(\Omega)$-norm for $1\leq p\leq \infty$. 
The letter $C$ denotes inessential constants which may vary from line to line. The dependence of $C$  upon parameters is indicated explicitly.

\vspace{2mm}

\section{Theory of Critical Manifolds}\label{sec2}

Under the change of variables  $u=e^w$, problem  \eqref{ieq1}-\eqref{consv} is reduced to
\begin{eqnarray}
& & \frac{\partial e^w}{\partial t}=\Delta w + \lambda\left(\frac{e^w}{\int_{\Omega}e^w\,dx}-\frac{1}{|\Omega|}\right), \quad x\in \Omega,\;\;t>0,\label{gnrf1}\\
& & w(x,0)=w_0(x), \quad x\in \Omega. 
 \label{gnrf2}
\end{eqnarray}
Integrating equation \eqref{gnrf1} over $\Omega$, taking also into account that $\partial\Omega=\emptyset,$  we obtain the total mass conservation
\[ 
\int_{\Omega} e^{w}\, dx=\int_{\Omega} e^{w_0}\, dx=\lambda. 
\] 
Hence it holds that 
\begin{eqnarray} 
&&\frac{\partial e^w}{\partial t}=\Delta w + e^w-\frac{\la}{|\Omega|}, \quad x\in \Omega,\;\;t>0\label{mmk1}\\
&&w(x,0)=w_0(x), \quad x\in \Omega,\quad\int_{\Omega} e^{w_0}\, dx=\lambda, \label{mmk2}
\end{eqnarray} 

and a related variational functional is
\bge
\mathcal{E}(w) =  \int_{\Om} \frac{1}{2} |\nabla w|^2- e^{w}+\frac{\la}{|\Om|}w \ dx,\quad w\in H^1(\Om)=V.  
\label{functional}
\ege
In relation to the Gel'fand triple
\[ 
V=H^1(\Om)\hookrightarrow X=L^2(\Om)\cong X^{*} \hookrightarrow V^{*}, 
\] 
the first variation of $\mathcal{E}(w)$ is given by
\bge\label{fv}
\delta \mathcal{E}(w)=-\Delta w-e^w+\frac{\la}{|\Om|}
\ege
and thus \eqref{mmk1}-\eqref{mmk2} is reduced to
\bge
&&\frac{\partial e^w}{\partial t}=-\delta \mathcal{E}(w), \quad x\in \Omega,\;\;t>0\label{dmmk1}\\
&&w(x,0)=w_0(x), \quad x\in \Omega,\quad\int_{\Omega} e^{w_0}\, dx=\lambda.\label{dmmk2}
\ege
Each steady-state $u^\ast \in F_\lambda$ to (\ref{ieq1})-(\ref{consv}) is a solution to (\ref{stationary}), and hence $w^*=\log u^\ast$  satisfies 
\[ \Delta w^\ast+e^{w^\ast}-\frac{\lambda}{\vert \Omega\vert}=0, \quad x\in \Omega, \] 
or equivalently,  
\[  
\delta \mathcal{E}(w^*)=0. 
\]  

This section is devoted to the proof of the following inequality, which casts a basis  for proving Theorem \ref{thm2} as in the standard theory of gradient inequality, cf. \cite{S83}. 

\begin{thm}\label{thm3}
Given $w^*\in V$ satisfying  $\delta \mathcal{E}(w^\ast)=0$, there exist $0<\theta\leq \frac{1}{2}$ and $\ve_0>0$ such that 
\begin{equation} 
w\in V, \ ||w-w^*||_{V}<\ve_0 \ \Rightarrow \ |\mathcal{E}(w)-\mathcal{E}(w^*)|^{1-\theta}\leq C||\delta \mathcal{E}(w)||_{V^*}.
 \label{theta} 
\end{equation} 
\end{thm}

To prove this result we decompose $\mathcal{E}(w)$ as in 
\[ 
\mathcal{E}(w)=\mathcal{E}_1(w)-\mathcal{E}_2(w) 
\] 
for 
\[ 
\mathcal{E}_1(w)= \int_{\Om} \frac{1}{2} |\nabla w|^2+\frac{\la}{|\Om|}w \ dx, \quad 
\mathcal{E}_2(w)= \int_{\Om} e^w\,dx. 
\] 
The first functional $\mathcal{E}_1:V\to \R$ is analytic, and it holds that 
\begin{eqnarray*} 
&&\delta\mathcal{E}_1(w_\ast)[w]=\int_{\Om} \nabla w\cdot \nabla w_\ast+\frac{\la}{|\Om|} \ dx\\
&& \delta^2\mathcal{E}_1(w_\ast)[w,w]=\int_{\Om} |\nabla w|^2\,dx=(\nabla w,\nabla w) \\
&& \delta^k\mathcal{E}_1(w_\ast)[\overbrace{w,w, \cdots, w}^{k}]=0, \ k\geq 3, \quad w,w_\ast \in V, 
\end{eqnarray*} 
where $( \ , \ )$ denotes the $L^2$-inner product. The second functional $\mathcal{E}_2:V\to \R$ is also analytic by the Trudinger-Moser-Fontana inequality \eqref{ftm}, which assures
\[ \sum_{k=0}^{\infty} \frac{1}{k!}\int_{\Om}e^{w_\ast} \vert w\vert^k\,dx = \int_{\Om} e^{w^\ast+|w|}\,dx<+\infty, \quad w,w_\ast\in V. \] 
Then
\[ 
\mathcal{E}_2(w+w_\ast)-\mathcal{E}_2(w_\ast)=\sum_{k=1}^{\infty} \frac{1}{k!}\int_{\Om} e^{w_\ast}w^k\,dx
\] 
and hence 
\[ \delta^k\mathcal{E}_2(w_\ast)[\overbrace{w,w, \cdots, w}^{k}]=\int_\Omega e^{w_\ast}w^k \ dx, \quad k\geq 1. \]

Given $w^\ast\in V$  satisfying $\delta \mathcal{E}(w^\ast)=0$, the linearized operator 
\bgee
\mathcal{L}\equiv\delta^2\mathcal{E}(w^*)=-\Delta-e^{w^*}: V\to V^* 
\egee
is realized as a self-adjoint operator in $X=L^2(\Om)$ with domain $D(\mathcal{L})=H^2(\Om)$. To develop the theory of critical manifold, cf.  \cite{Ch06}, we first introduce
\bgee
X_1\equiv \mbox{Ker} \ \mathcal{L}=\{v\in D(\mathcal{L})\big|\mathcal{L}v=0\}\subset V=H^1(\Omega). 
\egee
Let $dimX_1=n<\infty$ and $\langle \phi_1,\dots, \phi_n\rangle$ be an orthonormal basis of $X_1$, and  define the orthogonal projection  $\mathcal{P}: X\to X_1$, which can be extended to $\mathcal{P}:V^\ast \rightarrow X_1$, by 
\[ 
\mathcal{P}v=\sum_{i=1}^n (v, \phi_i)\phi_i=\sum_{i=1}^n \langle \phi_i, v\rangle_{V,V^\ast} \phi_i.
\] 

We next recall  the following theorem, derived from the implicit function theorem applied to 
\[ (I-\mathcal{P})\delta\mathcal{E}(v)=0. \] 
The local manifold $S$ defined by (\ref{216}) below is analytic because $\mathcal{E}:V\rightarrow \R$ is so. 

\begin{thm}[\cite{Ch06}]
Each $w^*\in V$  with  $\delta\mathcal{E}(w^\ast)=0$ admits a neighbourhood $U\subset V$ of $w^*$ 
such that 
\begin{equation} 
\mathcal{S}=\{w\in U\big| (I-\mathcal{P})\delta\mathcal{E}(w)=0\} 
 \label{216}
\end{equation} 
is a local analytic manifold around $w^\ast$ with dimension equal to  $n.$
 \label{ch1} 
\end{thm}

More precisely, we have the analytic mapping 
\bgee
g: U_1=U\cap X_1\to U_2=(I-\mathcal{P})U
\egee
such that $g(w_1^\ast)=w_2^\ast$ for $w^\ast=w_1^\ast +w_2^\ast \in U_1\oplus U_2$, and define
\[ \mathcal{S}=\{ w_1+ g(w_1) \mid w_1\in U_1\}. \] 
Then the following decomposition is valid
\begin{equation} 
w=w_1+w_2\in \mathcal{S}=U_1 \oplus U_2, \quad w_1=\mathcal{P}w, \ w_2=g(w_1), 
 \label{decomposition}
\end{equation} 
and then, the analytic mapping $Q:U\to \mathcal{S}$ is defined by 
\begin{equation} 
Qw=w_1+g(w_1)\in \mathcal{S}, \quad w=w_1+w_2\in  U_1 \oplus U_2.  
 \label{217}
\end{equation} 
We then  obtain 
\begin{equation} 
w-Qw=w_2-g(w_1)\in U_2 
 \label{218}
\end{equation} 
and also  
\begin{equation} 
Qw=w, \quad w\in \mathcal{S} 
 \label{219}
\end{equation} 
by (\ref{decomposition})-(\ref{217}). 

In the sequel we confirm several lemmas derived form the above structure.  

\begin{lem}\label{lem3}
It holds that 
\[ 
|\mathcal{E}(w)-\mathcal{E}(Qw)|\leq C||w-Qw||^2_V, \quad w\in U. 
\] 
\end{lem}

\begin{proof}
First, we have  
\begin{eqnarray} 
& & \mathcal{E}(w)-\mathcal{E}(Qw) \nonumber\\ 
& & \quad =\left\langle w-Qw, \delta \mathcal{E}(Qw)\right\rangle+\frac{1}{2}\delta^2\mathcal{E}(Qw)[w-Qw, w-Qw] +o\left(||w-Qw||_{V}^2\right).
 \label{11}
\end{eqnarray} 
Second, there arises 
\[ (I-\mathcal{P})(w-Qw)=w-Qw \] 
by (\ref{218}), and therefore, $Qw\in \mathcal{S}$ implies   
\begin{eqnarray*}   
\langle w-Qw, \delta \mathcal{E}(Qw)\rangle & = &  
\langle (I-\mathcal{P})(w-Qw) ,  \delta \mathcal{E}(Qw)\rangle \\ 
& = & \langle w-Qw,  (I-\mathcal{P})\delta \mathcal{E}(Qw)\rangle=0.  
\end{eqnarray*} 
Then \eqref{11} entails the desired  estimate
\bgee
|\mathcal{E}(w)-\mathcal{E}(Qw)|\leq C||w-Qw||^2_V.
\egee
 \end{proof}

 \begin{lem}\label{lem4}
Any $\ve>0$ admits $\delta>0$ such that 
\[ w\in V, \  ||w-w^*||_V<\delta \quad \Rightarrow \quad ||w-Qw||_V<\ve. \] 
 \end{lem}
 \begin{proof}
 We may assume $w\in U$.   Since $w^\ast\in \mathcal{S}$ it holds that $Qw^*=w^*$ by (\ref{219}), which implies 
\[ 
 ||w-Qw||_V \leq ||w-w^*||_{V}+||Qw^*-Qw||_V. 
 \] 
The result is now obvious because $Q:U\rightarrow \mathcal{S}$ is analytic. 
 \end{proof}
 
 \begin{lem}\label{lem5}
There is $\ve_0>0$ such that 
\bgee
w\in V, \ ||w-w^*||_V<\ve_0 \quad \Rightarrow \quad ||w-Qw||_V \leq C||\delta\mathcal{E}(w)||_{V^*}.
\egee
 \end{lem}
 
 \begin{proof}
 First, we have
 \begin{equation} 
 \delta\mathcal{E}(w)-\delta\mathcal{E}(Qw)=\delta^2\mathcal{E}(Qw)(w-Qw)+o\left(||w-Qw||_V\right).
  \label{q}
 \end{equation} 
 Since 
 \[ (I-\mathcal{P})\delta\mathcal{E}(Qw)=0 \] 
 by $Qw\in \mathcal{S}$, it follows that  
 \bge\label{13}
 (I-\mathcal{P})  \delta\mathcal{E}(w)= (I-\mathcal{P})\delta^2\mathcal{E}(Qw)(w-Qw)+o\left(||w-Qw||_V\right).
 \ege
 
Let 
$V_2=(I-\mathcal{P})(V)$ 
and recall $\mathcal{L}=\delta^2\mathcal{E}(w^*)$.  Then, 
\[ (I-\mathcal{P})\mathcal{L}:V_2\to V_2^* \] 
is an isomorphism. By Lemma \ref{lem4}, therefore, there is $\ve_1>0$ such that  
\[ (I-\mathcal{P})\delta^2\mathcal{E}({Qw}):V_2\to V_2^* \] 
is an isomorphism, provided that $||w-w^*||_V<\ve_1$ for $w\in V$.  

More precisely, we have $C_1>0$ such that  
\begin{equation} 
 ||z||_{V}\leq C_1 \left|\left|(I-\mathcal{P})\delta^2\mathcal{E}(Qw)(z)\right|\right|_{V^*}, \quad z\in V_2 
  \label{223}
\end{equation}  
for any $w\in V$ in $||w-w^*||_V<\ve_1$. Putting 
\[ z=w-Qw\in V_2=(I-\mathcal{P})V \] 
in (\ref{223}), we deduce
 \bgee
 ||w-Qw||_{V}&\leq& C_1 \left|\left|(I-\mathcal{P})\delta^2\mathcal{E}(Qw)(w-Qw)\right|\right|_{V^*}\\
 &=& C_1 \left|\left|(I-\mathcal{P})\delta\mathcal{E}(w)\right|\right|_{V^*}+o\left(||w-Qw||_V\right)
 \egee
 by (\ref{13}). Hence there is $\varepsilon_0>0$ such that 
 \[ w\in V, \ \Vert w-w^\ast\Vert_V<\varepsilon_0 \quad \Rightarrow \quad 
 ||w-Qw||_{V} \leq  C_2   \left|\left|(I-\mathcal{P})\delta\mathcal{E}(w)\right|\right|_{V^*}\leq C_3  \left|\left|\delta\mathcal{E}(w)\right|\right|_{V^*} 
\] 
by Lemma \ref{lem4}. 
 \end{proof}
 
 \begin{lem}\label{lem6}
 There is $\ve_0>0$ such that 
 \[ w\in V, \ ||w-w^*||_V<\ve_0 \quad \Rightarrow \quad ||\delta\mathcal{E}(Qw)||_{V^*}\leq C ||\delta\mathcal{E}(w)||_{V^*}. \] 
 \end{lem}
 \begin{proof}
Equality (\ref{q}), combined with $Qw\in \mathcal{S}$, implies 
\begin{eqnarray*} 
\delta\mathcal{E}(Qw)&=& \mathcal{P} \delta \mathcal{E}(Qw)\\
&=& \mathcal{P}\delta \mathcal{E}(w)- \mathcal{P}\delta^2 \mathcal{E}(Qw)(w-Qw)+o\left(||w-Qw||_V\right). 
\end{eqnarray*} 
Hence it follows that 
\begin{equation} 
||\delta\mathcal{E}(Qw)||_{V^*}\leq ||\delta\mathcal{E}(w)||_{V^*}+C_4||w-Qw||_V+o\left(||w-Qw||_V\right). 
 \label{225}
\end{equation} 
Then Lemma \ref{lem4} assures there exists $\varepsilon_1>0$ such that 
\[ w\in V, \ ||w-w^*||_V<\ve_1 \quad \Rightarrow \quad ||\delta\mathcal{E}(Qw)||_{V^*} \leq ||\delta\mathcal{E}(w)||_{V^*}+C_5||w-Qw||_V. \] 
We finally obtain $\varepsilon_0>0$ such that 
\[  w\in V, \ ||w-w^*||_V<\ve_0 \quad \Rightarrow \quad ||\delta\mathcal{E}(Qw)||_{V^*}\leq C_5  ||\delta\mathcal{E}(w)||_{V^*} \] 
by Lemma \ref{lem5}.
 \end{proof}
 
 \begin{lem}\label{lem7}
 There exist $0<\theta\leq \frac{1}{2}$ and $\ve_0>0$ such that 
\begin{equation}\label{14}
 w\in \mathcal{S}, \ \Vert w-w^\ast\Vert_V<\varepsilon_0 \quad \Rightarrow \quad 
 |\mathcal{E}(w)-\mathcal{E}(w^*)|^{1-\theta}\leq C||\delta \mathcal{E}(w)||_{V^*}.
 \end{equation} 
  \end{lem}
  
 \begin{proof}
 Since $\mathcal{S}$ is a finite dimensional analytic manifold and $\mathcal{E}: S\to \R$ is analytic, the result follows from Lemma \ref{cgi}. 
 \end{proof}
 
We are ready to give the proof of Theorem \ref{thm3}. 

\begin{proof}[Proof of Theorem \ref{thm3}]

Given $w\in V$ in $||w-w^*||_V\ll 1$, we have   
\bgee
|\mathcal{E}(w)-\mathcal{E}(w^*)|&\leq& |\mathcal{E}(w)-\mathcal{E}(Qw)|+|\mathcal{E}(Qw)-\mathcal{E}(w^*)|\\
&\leq& C(||w-Qw||_V^2+||\delta\mathcal{E}(Qw)||_{V^*}^{\frac{1}{1-\theta}})
\egee
by Lemma \ref{lem3}, Lemma \ref{lem7}, and $Qw\in \mathcal{S}$, where $0<\theta\leq \frac{1}{2}$. Then Lemmas \ref{lem5} and \ref{lem6} imply 
\bgee
|\mathcal{E}(w)-\mathcal{E}(w^*)|\leq C(||\delta\mathcal{E}(w)||_{V^*}^{2}+||\delta\mathcal{E}(w)||_{V^*}^{\frac{1}{1-\theta}})
\egee
and hence the desired property \eqref{14}.
\end{proof}

\section{Proof of Theorem \ref{thm2}}\label{sec3}

Note that assuming $0<\lambda\leq 8\pi$ in (\ref{mmk1})-(\ref{mmk2}), we have readily obtained that $T=+\infty$ and the orbit $\mathcal{O}=\{ w(\cdot,t)\}$ is pre-compact in $C(\Omega)$ by Theorem \ref{thm1}. To apply Theorem \ref{thm3}, we use the parabolic regularity in the following form. 

\begin{lem}\label{lem8}
Given $w^*\in V$  with $\delta\mathcal{E}(w^*)=0$, we obtain 
\begin{equation} 
\sup_{t_0\leq t<t_0+T} ||w(\cdot, t)-w^*||_V\leq C(||w(\cdot, t_0)-w^*||_V+\sup_{t_0\leq t<t_0+T} ||w(\cdot, t)-w^*||_2),
 \label{c1}
\end{equation} 
for any $t_0\geq 0$ and $T>0$. 
\end{lem}
\begin{proof}
Since 
\[ w_t=e^{-w}\Delta w+1-\frac{\la}{|\Omega|}e^{-w}, \quad 0=e^{-w^\ast}\Delta w^\ast+1-\frac{\la}{|\Omega|}e^{-w^\ast} \] 
the function $z=w-w^*$ solves
\bgee
z_t&=&e^{-w} \Delta z+\left(e^{-w}-e^{-w^*}\right)\Delta w^*-\frac{\la}{|\Om|}\left(e^{-w}-e^{-w^*}\right)\\
&=& \left(e^{-w}\Delta-1\right)z+bz 
\egee
with $b=b(x,t)$ uniformly bounded. Here, $e^{-w}\Delta$ generates an evolution operator, denoted by $\{ U(t,s)\}$, satisfying 
\begin{equation} 
\Vert U(t,s)z_0\Vert_V\leq C \Vert z_0\Vert_V, \quad \Vert U(t,s)z_0||_{V}\leq C (t-s)^{-1/2}\Vert z_0\Vert_2, \quad 0\leq s<t<\infty. 
 \label{131} 
\end{equation} 
Hence $z(t)=U(t,z)z_0$ is the solution to 
\[ z_t=e^w\Delta z \quad \mbox{in $\Omega\times (s,+\infty)$}, \quad \left. z\right\vert_{t=s}=z_0. \] 
If $\widetilde{U}(t,s)$ denotes the evolution operator associated with $e^{-w}\Delta-1$, therefore, it holds that 
\bgee
||\widetilde{U}(t,s)||_{V\to V}\leq C e^{-(t-s)}, \quad ||\widetilde{U}(t,s)||_{X\to V}\leq C (t-s)^{-1/2} e^{-(t-s)}, \quad 0\leq s<t<\infty, 
\egee
and furthermore, 
\bgee
z(t)=\widetilde{U}(t,t_0)z(t_0)+\int_{t_0}^t \widetilde{U}(t,r) (bz)(r)\,dr, \quad t\geq t_0. 
\egee
Thus we obtain 
\bgee
||z(t)||_V\leq C\left( ||z(t_0)||_V+\int_{t_0}^t (t-r)^{-1/2} e^{-(t-r)}\,dr \sup_{t_0\leq r< t} ||z(r)||_2\right)
\egee
which finally entails
\bgee
\sup_{t_0\leq t<t_0+T} ||z(t)||_V&\leq& C\left(||z(t_0)||_V +\int_0^{\infty} s^{-1/2} e^{-s}\,ds \sup_{t_0\leq t<t_0+T} ||z(t)||_2\right)\\
&\leq&  C(||z(t_0)||_V + \sup_{t_0\leq t<t_0+T} ||z(t)||_2). 
\egee
\end{proof}

\begin{rem}
The second inequality of (\ref{131}) implies 
\[ \Vert w(\cdot,t+1)-w^\ast\Vert_V\leq C\sup_{t\leq s <t+1}\Vert w(\cdot,s)-w^\ast\Vert_2 \] 
for any $t\geq 0$ by 
\[ z(t+1)=\widetilde{U}(t+1,t)z(t)+\int_t^{t+1}\widetilde{U}(t+1,r)(bz)(r) \ dr. \] 
 \label{rempr}
\end{rem}

Now we are ready to give the proof of the main result in the current section. 

\begin{proof}[Proof of Theorem \ref{thm2}] 

We prescribe the constant $C$ in (\ref{c1}) as  $C=C_1\geq 1$ and thus:
\begin{equation} 
\sup_{t_0\leq t<t_0+T} ||w(\cdot, t)-w^*||_V\leq C_1(||w(\cdot, t_0)-w^*||_V+\sup_{t_0\leq t<t_0+T} ||w(\cdot, t)-w^*||_2).
 \label{c2}
\end{equation} 
Let the $\omega$-limit set of (\ref{mmk1})-(\ref{mmk2}) be 
\[ 
\omega(w_0)=\{w^*\in V \mid \mbox{there exists $t_k\rightarrow\infty$ such that 
$w(t_k)\to w^*$ in $C^{\infty}$ topology} \}. 
\]  
By Theorem \ref{thm11}, this $\omega(w_0)$ is non-empty, compact, connected, and satisfies  
\[ 
\omega(w_0) \subset \{w^*\in V\mid \delta\mathcal{E}(w^*)=0\}.
\] 
Hence we have $w^*\in V$ with $\delta\mathcal{E}(w^*)=0$ and $t_k\rightarrow \infty$ such that  
\begin{equation} 
w(\cdot,t_k)\rightarrow w^\ast \quad \mbox{in $C^\infty$ topology} 
 \label{accumulate}
\end{equation} 
and in particular, 
\begin{equation}\label{15a}
\Vert w(\cdot,t_k)-w^*\Vert_V\leq \frac{\ve_0}{4C_1}, \quad\mbox{for}\quad\mbox k\gg 1,  
\end{equation} 
where $\varepsilon_0>0$ and $C_1\geq 1$ are constants prescribed in Theorem \ref{ch1} and (\ref{c2}), respectively.

We have
\bgee
\frac{d}{dt} \mathcal{E}(w)=-\langle w_t,\delta\mathcal{E}(w)\rangle_{V,V^*}=-(w_t, e^w w_t) \leq 0,
\egee
by (\ref{mmk1}), and hence the existence of 
\bge\label{16}
\lim_{t\to \infty}\mathcal{E}(w(\cdot, t))=\mathcal{E}_{\infty}=\mathcal{E}(w^*), 
\ege
where the second equality follows from $w^*\in \omega(w_0)$. In particular, 
\[ \mathcal{H}(t)=(\mathcal{E}(w(\cdot, t))-\mathcal{E}(w^*))^{\theta}\geq 0 \] 
is well-defined, and it holds that 
\begin{equation}
\lim_{t\to \infty}\mathcal{H}(t)=0.
 \label{16a}
\end{equation} 

Since 
\[ C_2^{-1}\leq e^w\leq C_2 \quad \mbox{in $\Om\times (0,\infty)$} \] 
is valid with $C_2\geq 1$, we obtain 
\begin{eqnarray*} 
-\frac{d \mathcal{H}}{dt}&=&-\theta \left(\mathcal{E}(w)-\mathcal{E}(w^*)\right)^{\theta-1} \left<w_t, \delta\mathcal{E}(w)\right>_{V,V^*}\\
&=& \theta \left(\mathcal{E}(w)-\mathcal{E}(w^*)\right)^{\theta-1} (e^w, w_t^2) \\
&\geq& \theta C_2^{-1}  \left(\mathcal{E}(w)-\mathcal{E}(w^*)\right)^{\theta-1}  ||w_t||_2^2\\
&\geq& \theta C_2^{-3/2}  \left(\mathcal{E}(w)-\mathcal{E}(w^*)\right)^{\theta-1}||w_t||_2 \left(\int_{\Om} e^w w_t^2\,dx\right)^{1/2}\\
&=&  \theta C_2^{-3/2}  \left(\mathcal{E}(w)-\mathcal{E}(w^*)\right)^{\theta-1}||w_t||_2 ||\delta\mathcal{E}(w)||_2  
\end{eqnarray*} 
again by (\ref{mmk1}). Therefore, there is $C_3>0$ such that
\bge\label{17}
-\frac{d \mathcal{H}}{dt}\geq \frac{1}{C_3} \left(\mathcal{E}(w)-\mathcal{E}(w^*)\right)^{\theta-1}||w_t||_2 ||\delta\mathcal{E}(w)||_{V^*}.
\ege

To apply Theorem \ref{thm3}, assume the existence of $t_0>t_k$ such that 
\begin{equation} 
\Vert w(\cdot,t)-w^\ast\Vert_V <\varepsilon_0, \quad t_k\leq t\leq t_0. 
 \label{135}
\end{equation} 
Then inequality (\ref{17}) implies  
\begin{equation} 
\Vert w_t\Vert_2\leq -C_4\frac{d\mathcal{H}}{dt}, \quad t_k\leq t\leq t_0,  
 \label{236}
\end{equation} 
where $C_4>0$ is a constant. It follows that  
\[ 
\Vert w(\cdot, t)-w(\cdot, t_k)\Vert_2\leq C_4 \mathcal{H}(t_k), \quad t_k\leq t\leq t_0,
\] 
and thus we obtain 
\begin{equation} 
\Vert w(\cdot, t)-w(\cdot, t_k)||_V\leq \frac{\ve_0}{4}+C_1 C_4 \mathcal{H}(t_k)  \quad t_k\leq t\leq t_0 
 \label{136}
\end{equation} 
by (\ref{c2}) and (\ref{15a}) with $C_1\geq 1$. 
 
Equality \eqref{16a} assures $k\gg1$ satisfying 
\begin{equation}  
\mathcal{H}(t_k)<\frac{\ve_0}{4C_1C_4}. 
 \label{137}
\end{equation} 
Fix such $k$. By the above argument, if there is $t_0>t_k$ provided with (\ref{135}), it holds that (\ref{136}) and hence 
\begin{equation}  
\Vert w(\cdot,t)-w(\cdot,t_k)\Vert_V <\varepsilon_0/2, \quad t_k\leq t\leq t_0 
 \label{1381}
\end{equation} 
by (\ref{137}).  Since we have readily assumed (\ref{15a}) with $C_1\geq 1$, inequality (\ref{1381}) implies 
\begin{equation}  
\Vert w(\cdot,t)-w^\ast\Vert_V <3\varepsilon_0/4, \quad t_k\leq t\leq t_0.  
 \label{139}
\end{equation} 
We have thus observed that (\ref{135}) implies (\ref{139}). Regarding (\ref{15a}) with $C_1\geq 1$ again, we conclude  
\begin{equation} 
\Vert w(\cdot,t)-w^\ast\Vert_V <\varepsilon_0, \quad t\geq t_k. 
 \label{138}
\end{equation} 
Consequently,  by (\ref{138})  inequality (\ref{236}) is improved as 
\bge\label{17b}
||w_t||_2\leq -C_4\frac{d \mathcal{H}}{dt}, \quad t\geq t_k,
\ege
which implies 
\[ 
\int_0^{\infty} ||w_t||_2 dt<\infty.
\] 
Then we obtain  
\[ \lim_{t\to \infty} \Vert w(\cdot, t)-w^*\Vert_2=0  \] 
by (\ref{accumulate}), and hence $\omega(w_0)=\{w^*\}$ from the uniqueness of the limit. It thus follow that 
\begin{equation} \label{18}
w(\cdot, t)\to w^*, \ t\rightarrow \infty \quad\mbox{in $C^{\infty}$ topology}.
\end{equation} 

Turning to the rate of convergence, we use   
\[ \vert \mathcal{E}(w(\cdot,t))-\mathcal{E}(w^\ast)\vert^{1-\theta}\leq C_5\Vert \delta \mathcal{E}(w(\cdot,t))\Vert_{V^\ast}, \quad t\geq t_k \] 
derived from Theorem \ref{thm3} and (\ref{138}).  Using  (\ref{mmk1}) again we derive 
\begin{eqnarray*} 
-\frac{d \mathcal{H}}{dt} & = & \theta(\mathcal{E}(w)-\mathcal{E}(w^\ast))^{\theta-1}\langle w_t, -\delta \mathcal{E}(w) \rangle \\ 
& = & \theta (\mathcal{E}(w)-\mathcal{E}(w^\ast))^{\theta-1}\Vert \delta \mathcal{E}(w)\Vert_2^2 \\ 
&\geq& \frac{1}{C_6} \left(\mathcal{E}(w)-\mathcal{E}(w^*)\right)^{\theta-1}||\delta\mathcal{E}(w)||_{V^*}^2 \\
&\geq& \frac{1}{C_6C_5^2} \left(\mathcal{E}(w(t))-\mathcal{E}(w^*)\right)^{1-\theta}\\
&=&   \gamma \mathcal{H}^{\frac{1}{\theta}-1},  \ \gamma=\frac{1}{C_6C_5^2}, \quad \quad t\geq t_k. 
\end{eqnarray*} 
We thus obtain  
\[ 
\mathcal{H}(t)\leq C \Phi(t), \quad t\geq t_k,
\] 
where
\[ 
\Phi(t)=\left\{ \begin{array}{ll}
t^{-\frac{\theta}{1-2\theta}}, & 0<\theta<\frac{1}{2}\\
e^{-\gamma t}, & \theta=\frac{1}{2}. \end{array} \right. 
\] 

Inequality \eqref{17b} now implies 
\bgee
||w(\cdot, t)-w(\cdot, s)||_2\leq C \Phi(s), \quad t\geq s\geq t_k, 
\egee
and sending $t\to \infty$, we get  
\[ 
\Vert w^*-w(\cdot,s) \Vert|_2 \leq C\Phi(t), \quad s\geq t_k, 
\] 
or
\begin{equation} 
\Vert w(\cdot,t)-w^*\Vert_2 \leq C\Phi(t), \quad t\geq t_k, 
 \label{1315}
\end{equation} 
Then, Remark \ref{rempr} entails
\bgee
||w(\cdot, t)-w^*||_V \leq C\Phi(t), \quad t\to \infty. 
\egee

Given multi-index $\alpha$, we can derive an equation of $z_\alpha=D^\alpha(w-w^\ast)$, where the second estimate of (\ref{131}) is applicable. Then an iteration ensures the rate of convergence $\Phi(t)$ in (\ref{18}).  
\end{proof}

\section{Non-degenerate Steady-States}\label{sec4}

Recall that $u^\ast=u^\ast(x)>0$ is called a steady-state to (\ref{ieq1})-(\ref{consv}) when it solves (\ref{stationary}). Then $v^\ast=v^\ast(x)$ defined by (\ref{transform}) satisfies (\ref{113}), which is the Euler-Lagrange equation for the functional $J_\lambda=J_\lambda(v)$ of $v\in V_0$, defined by (\ref{functionalj})-(\ref{functionspace}). We say that $u^\ast$ is non-degenerate if this $v^\ast\in V_0$ is a non-degenerate critical point of $J_\lambda$ on $V_0$. Here, $u^\ast$ is reproduced by $v^\ast$ through (\ref{0114}). 

More precisely, first, we notice  
\[ 
\delta J_{\la}(v)[\phi] = \left. \frac{d}{ds}J_\lambda(v+s\phi)\right\vert_{s=0} = (\nabla\phi, \nabla\phi)-\frac{\lambda\int_\Omega e^v\phi \ dx}{\int_\Omega e^v \ dx}, \quad \phi \in V_0, \] 
to identify   
\[ 
\delta J_{\la}(v)=-\Delta v-\lambda\left(\frac{e^{v}}{\int_{\Om} e^{v}\,dx}-\frac{1}{\vert \Omega\vert}\right)\in V_0^\ast, \quad v\in V_0.  \] 
Hence the above $v^*$, realized as a solution to (\ref{113}), belongs to $V_0$ and is a critical point of $J_{\la}$ on $V_0$. 

Second, the quadratic form $\mathcal{Q}:V_0\times V_0\rightarrow \R$ defined by 
\begin{eqnarray*} 
\mathcal{Q}(\phi, \phi) & = & \left.\frac{d^2}{ds^2}J_\lambda(v^\ast+s\phi)\right\vert_{s=0} \nonumber\\ 
& = & (\nabla\phi, \nabla\phi) -\lambda\frac{\int_{\Om}e^{v^*}\phi^2 \ dx}{\int_\Omega e^{v^\ast}} + \lambda\left( \frac{\int_\Omega e^{v^\ast}\phi \ dx}{\int_\Omega e^{v^\ast} \ dx}\right)^2   
\end{eqnarray*} 
is associated with the linearized operator $\delta^2 J_{\la}(v^*):V_0\rightarrow V_0^\ast$ through 
\[ \mathcal{Q}(\phi,\phi)=\langle \phi, \delta^2J_\lambda(v^\ast)\phi \rangle_{V,V^*}. \] 
This $\delta^2J_\lambda(v^\ast)$ is realized as a self-adjoint operator in $X_0=L^2(\Omega)\cap V_0$, denoted by $\mathcal{B}$, with the domain $D(\mathcal{B})=H^2(\Om)\cap V_0$, satisfying 
\[ (\mathcal{B}\phi, \psi)=\mathcal{Q}(\phi,\psi), \quad \phi \in D(\mathcal{B})\subset V_0, \ \psi\in V_0.  \]  
Hence it holds that 
\begin{eqnarray} 
\mathcal{B}\phi & =&-\Delta \phi-\frac{\la e^{v^*}}{\int_{\Om} e^{v^*}\,dx}\phi+\frac{\la \int_{\Om}e^{v^*}\phi \,dx}{\left(\int_{\Om} e^{v^*}\,dx\right)^2} \ e^{v^*} \nonumber\\
&=&-\Delta \phi-u^*\phi+\frac{1}{\la} (\phi, u^*)u^\ast,  \quad \phi\in D(\mathcal{B})=H^2(\Omega)\cap V_0  
 \label{42}
\end{eqnarray} 
by (\ref{0114}). 

Now we show the following lemma stated in Section \ref{sec1}.

\begin{lem}\label{lem0}
The stationary solution $u^*=u^*(x)>0$ to \eqref{ieq1}-\eqref{consv} is non-degenerate if and only if the  property (\ref{ikk39}) holds.  
\end{lem}

\begin{proof}
By the definition, the non-degeneracy of $u^\ast$ means the non-dgeneracy of  $\mathcal{B}$ in $X_0=L^2(\Omega)\cap V_0$, which is equivalent to  
\begin{equation} \label{mk46}
\phi\in D(\mathcal{B}), \ \mathcal{B}\phi=0 \quad \Rightarrow \quad \phi=0. 
\end{equation} 

Assume, first, $\phi\in D(\mathcal{B})\setminus\{0\}$ with $\mathcal{B}\phi=0$, and let  
\[ \psi=\phi-\frac{1}{\la}\int_{\Om} u^*\phi\,dx \in H^2(\Omega). \] 
Then we have  
\[  -\Delta \psi=u^\ast\psi, \quad \int_\Omega \psi u^\ast \ dx=0. \] 
It also holds  that $\psi\neq 0$ by $\phi\in V_0\setminus \{0\}$. Hence if $u^\ast$ is degenerate there is $\psi\in V\setminus\{0\}$ satisfying (\ref{ikk39}). 

If problem \eqref{ikk39} admits $\psi\in H^2(\Om)\setminus\{0\}$, second, we take 
\[ \phi=\psi-\frac{1}{\vert \Omega\vert}\int_{\Om} \psi\,dx\in H^2(\Omega)\cap V_0=D(\mathcal{B}). \] 
It holds that 
\[ (\phi, u^*)=-\frac{\la}{|\Om|} \int_{\Om} \psi\,dx, \] 
and hence 
\begin{eqnarray*} 
\mathcal{B}\phi & = & -\Delta \phi-u^\ast\phi+\frac{1}{\lambda}(\phi, u^\ast)u^\ast \\ 
& = & -\Delta \psi-u^\ast\psi+u^\ast\frac{1}{\vert\Omega\vert}\int_\Omega \psi+\frac{1}{\lambda}(\phi, u^\ast)u^\ast \\ 
& = & -\Delta \psi-u^\ast\psi=0 
\end{eqnarray*} 
by (\ref{ikk39}). If $\phi=0$, then it holds that $\psi=\mbox{constant}$ and by virtue of  (\ref{ikk39}), there arises $\psi=0$, a contradiction. Thus $\mathcal{B}$ has the eigenvalue $0$, and hence this operator is degenerate. 
\end{proof}

\begin{lem}\label{lem0a}
Let   $u^*=u^*(x)>0$ be a steady-state to \eqref{ieq1}-\eqref{consv}, and define $w^\ast\in V=H^1(\Omega)$ by (\ref{transform}), i.e., 
\begin{equation} 
w^\ast=\log u^\ast. 
 \label{441}
\end{equation} 
Set
\begin{equation} 
\mathcal{M}=-\Delta -u^*:V\to V^*. 
 \label{operator}
\end{equation} 
Then the following statements are equivalent. 
\begin{enumerate}
\item
There exists $C>0$ such that 
\begin{equation} 
\phi\in V, \ \int_{\Om} u^*\phi\,dx=0 \quad \Rightarrow \quad 
||\phi||_V\leq C||\mathcal{M}\phi||_{V^*}. 
 \label{psk1}
\end{equation} 
\item
There exists $\ve_0>0$ and $C>0$ such that 
\begin{equation} 
w\in V, \ \int_{\Om} e^w\,dx=\la, \ ||w-w^*||_V<\ve_0 \quad 
\Rightarrow \quad 
||w-w^*||_{V}\leq C|| \mathcal{M}(w-w^*)||_{V^*}. 
 \label{psk2}
\end{equation} 
\end{enumerate}
\end{lem}

\begin{proof}
$(i)\implies (ii)$: \quad Assume (i), take  
\begin{equation}  
w\in V, \ \int_{\Om} e^w\,dx=\la, \ ||w-w^*||_V<\ve_0,  
 \label{46}
\end{equation} 
and let 
\[ \phi^*=\frac{u^*}{||u^*||_2}, \quad z=w-w^*, \quad \mathcal{P}z=z-(\phi^*,z)\phi^*. \] 
It holds that $(\mathcal{P}z, u^*)=0$, and hence 
\bge\label{psk3}
\left|\left|\mathcal{P}z\right|\right|_V\leq C \left|\left|\mathcal{M}(\mathcal{P}z)\right|\right|_{V^*} 
\ege
by \eqref{psk1}. Then, we obtain  
\begin{eqnarray}\label{psk4}
||z||_{V}&\leq& \left|\left|\mathcal{P}z\right|\right|_V+|(\phi^*,z)|\;||\phi^*||_V\nonumber\\
&\leq& C \left|\left|\mathcal{M}(\mathcal{P}z)\right|\right|_{V^*}+|(\phi^*,z)| ||\phi^*||_V\nonumber\\
& = & C\Vert \mathcal{M}z-(u^\ast, z)\Vert u^\ast\Vert_2^{-1}\mathcal{M}\phi^\ast\Vert_{V^\ast}+\vert (u^\ast, z)\vert \Vert \phi^\ast\Vert_V\cdot \Vert u^\ast\Vert_2^{-1} \nonumber\\ 
&\leq& C \left|\left|\mathcal{M}(z)\right|\right|_{V^*}+\vert (u^*,z)\vert (C \Vert \mathcal{M}\phi^*\Vert_{V^*}+\Vert \phi^*\Vert_V)\Vert u^*\Vert_2^{-1}.
\end{eqnarray} 

Here we have   
\[ \int_\Omega e^w \ dx=\int_\Omega e^{w^\ast} \ dx=\lambda \] 
by (\ref{441}), $\Vert u^\ast\Vert_1=\lambda$, and  (\ref{46}). Hence it holds that  
\begin{equation} 
0=\int_0^1 \int_{\Om} e^{sw+(1-s)w^*}(w-w^*)\,dx \,ds   
 \label{4410}
\end{equation} 
by 
\[ 
e^w-e^{w^*}=\int_0^1 \frac{d}{ds} e^{sw+(1-s)w^*}\,ds=\int_0^1  e^{sw+(1-s)w^*}(w-w^*)\,ds. 
\] 
Then (\ref{4410}) implies  
\begin{eqnarray} 
(u^*,z) & = & \int_{\Om} e^{w^*} (w-w^*)\,dx \nonumber\\ 
& = & \int_0^1 \int_{\Om}(e^{w^*}-e^{sw+(1-s)w^*})(w-w^*)\,dx\,ds.\nonumber\\ 
& = & \int_0^1 \int_{\Om}(e^{w^*}-e^{sw+(1-s)w^*})z \,dx\,ds.
 \label{411} 
\end{eqnarray} 

In (\ref{411}) we have 
\bgee
e^{w^*}-e^{sw+(1-s)w^*}&=&\int_0^1 \frac{d}{dr} e^{rw^*+(1-r)(sw+(1-s)w^*)}\,dr\\
&=& \int_0^1  e^{rw^*+(1-r)(sw+(1-s)w^*)} (-s)z\,dr, 
\egee
and therefore, 
\[ 
\vert (u^*,z)\vert \leq \int_0^1\int_0^1 \int_{\Om}   e^{rw^*+(1-r)(sw+(1-s)w^*)} s z^2\,dx\,dr\,ds.
\] 
By the Trudinger-Moser-Fontana inequality, any $K>0$ admits $C_1(K)$ such that 
\[ \Vert w\Vert_V\leq K \quad \Rightarrow \quad  \Vert e^{rw^*+(1-r)(sw+(1-s)w^*)}\Vert_2\leq C_1(K), \ 0\leq r,s\leq 1, \] 
and hence we find $C_2(K)>0$ such that 
\bge\label{psk5}
|(u^*,z)|\leq C_1(K)||z||_4^2\leq C_2(K) ||z||_V^2.
\ege

Combining \eqref{psk4} and \eqref{psk5}, we reach to
\[ \Vert z\Vert_V\leq C_3(K)(\Vert \mathcal{M}z\Vert_V+\Vert z\Vert_V^2) \] 
for $\Vert z\Vert_V\leq K$. Then (\ref{psk2}) folllows for $\varepsilon_0=\frac{1}{2C_0(K)}$, because then we have 
\[ \Vert z\Vert_V=\Vert w-w^\ast\Vert_V <\varepsilon_0 \quad \Rightarrow \quad C_3(K)\Vert z\Vert_V^2\leq \frac{1}{2}\Vert z\Vert_V \] 
and hence (\ref{psk2}) with $C=2C_3(K)$. 

\vspace{4mm} 

$(ii)\implies (i)$: \quad Given  
\begin{equation} 
\phi \in V, \ \int_{\Om} u^*\phi\ dx=0, 
 \label{413}
\end{equation} 
we show the conclusion of (\ref{psk1}): 
\begin{equation} 
\Vert \phi\Vert_V\leq C\Vert \mathcal{M}\phi\Vert_{V^\ast}. 
 \label{conclusion}
\end{equation} 
For this purpose, it suffices to assume 
\begin{equation} 
\phi\neq 0. 
 \label{nonzero}
\end{equation} 
Define
\[ 
\Phi(s,z)=\left\{ \begin{array}{ll}
\frac{1}{s} \int_{\Om} e^{s\phi+s^2 z+w^*}-e^{w^*}\,dx, & s\neq 0 \\
0, & s=0,  \end{array} \right. \quad (s,z)\in \R\times V.  
\]  
Note that
\begin{eqnarray*} 
e^{s\phi+s^2z+w^\ast}-e^{w^\ast} & = & e^{w^\ast}(e^{s\phi+s^2z}-1) \\ 
& = & e^{w^\ast}\{ (s\phi+s^2z)+\frac{1}{2}(s\phi+s^2z)^2+o(s^2)\} \\ 
& = & \{ s\phi+s^2(z+\frac{1}{2}\phi^2)\} e^{w^\ast}+o(s^2), \quad s\rightarrow 0,  
\end{eqnarray*} 
to deduce
\[ \Phi(s,z)=\int_\Omega \{ \phi+s(z+\frac{1}{2}\phi^2)\}e^{w^\ast} \ dx+o(s), \quad s\rightarrow 0. \] 
First, this $\Phi=\Phi(s,z)$ is continuous in $(s,z)\in \R\times V$ because 
\[ \lim_{s\rightarrow 0}\Phi(s,z)=0 \] 
follows from (\ref{441}) and (\ref{413}): 
\begin{equation} 
\int_\Omega e^{w^\ast}\phi \ dx=0. 
 \label{413x}
\end{equation} 
Second, the following  limit arises   
\[  
\lim_{s\rightarrow 0}\Phi_s(s,z)=\int_\Omega (z+\frac{1}{2}\phi^2)e^{w^\ast} \ dx
\] 
and hence  $\Phi$ is $C^1$ in $\R\times V$. It holds, in particular, that     
\[ 
\Phi_s(0,0)=\frac{1}{2}\int_{\Om} e^{w^*}\phi^2\,dx\neq 0 
\] 
by (\ref{nonzero}), and therefore, the implicit function theorem guarantees the existence of a $C^1$ function $z=z(s)$ of $s$ such that 
\[ z(0)=0, \quad \Phi(s,z(s))=0, \ \vert s\vert\ll 1. \] 
Accordingly, 
\[ w(s)=s\phi+s^2z(s)+w_\ast \] 
satisfies 
\begin{equation} 
w(0)=w_\ast, \quad \dot w(0)=\phi, \quad \int_\Omega e^{w(s)} \ dx=\int_\Omega e^{w^\ast} \ dx=\lambda, \ \vert s\vert \ll 1 
 \label{2414}
\end{equation} 
and hence 
\begin{equation} 
\Vert w(s)-w^*\Vert _V\leq C\Vert |\mathcal{M}(w(s)-w^*)\Vert_{V^*}, \quad |s|\ll 1
 \label{2415}
\end{equation} 
by (\ref{psk2}). Then, (\ref{conclusion}) follows from (\ref{2414})-(\ref{2415}).
\end{proof}

\section{Proof of Theorem \ref{thm14}}\label{sec5}

Given a non-degenerate steady-state $u^\ast=u^\ast(x)>0$ of (\ref{ieq1})-(\ref{consv}), define $w^\ast\in V$ by (\ref{441}). Then it holds that  
\begin{equation} 
\delta \mathcal{E}(w^\ast)=0, \quad \int_\Omega e^{w^\ast} \ dx=\lambda. 
 \label{511}
\end{equation} 
By Lemma \ref{lem0}, the operator $\mathcal{M}:V\rightarrow V^\ast$ defined by (\ref{operator}) is provided with the property (\ref{psk1}). Then we obtain $\varepsilon_0>0$ satisfying (\ref{psk2}) by Lemma \ref{lem0a}.  

Having these properties, we see that Theorem \ref{thm14} is reduced to the following lemma by the proof of Theorem \ref{thm2}. 

\begin{lem}
Let $w^\ast\in V$ satisfy (\ref{511}), and assume the propery (\ref{psk2}). Then, there arises that  $\theta=\frac{1}{2}$ in the conclusion of (\ref{theta}) for $w$ satisfying 
\begin{equation} 
w\in V, \quad \Vert w-w^\ast\Vert_V<\varepsilon_1, \quad \int_\Omega e^w \ dx=\lambda  
 \label{55}
\end{equation} 
for $\varepsilon_1>0$ sufficiently small. 
 \label{lem51}
\end{lem}

For the proof of this lemma, we first verify several facts derived from the Trudinger-Moser-Fontana inequality

 \begin{lem}
 Any $K>0$ admits $C(K)>0$ such that 
\begin{equation} 
 w_1,w_2\in V, \ ||w_1||_V, ||w_2||_V\leq K \ \ \Rightarrow \ \ 
 ||\delta\mathcal{E}(w_1)-\delta\mathcal{E}(w_2)||_{V^*}\leq C(K)||w_1-w_2||_V.
  \label{418}
\end{equation}  
 \end{lem}
 \begin{proof}
Given $w\in V=H^1(\Omega)$, let 
\[ \overline{w}=\frac{1}{\vert \Omega\vert}\int_\Omega w \ dx, \quad [w]=w-\overline{w}\in V_0. \] 
Take $z\in V$ then  we have
\begin{equation} 
\langle z, \delta\mathcal{E}(w_1)-\delta\mathcal{E}(w_2)\rangle_{V,V'}
=\int_{\Om} \nabla z \cdot \nabla(w_1-w_2)-z(e^{w_1}-e^{w_2}) \,dx 
 \label{213}
\end{equation} 
by \eqref{fv}, where 
\begin{eqnarray*}
e^{w_1}-e^{w_2} & = & \int_0^1 \frac{d}{ds} e^{sw_1+(1-s)w_2}\,ds=\int_0^1  e^{sw_1+(1-s)w_2}\,ds \cdot (w_1-w_2) \nonumber\\ 
& = & \int_0^1e^{s\overline{w_1}+(1-s)\overline{w_2}}\cdot e^{[sw_1+(1-s)w_2]} \ ds \cdot (w_1-w_2). 
\end{eqnarray*} 
Hence it follows that 
\begin{equation} 
\vert e^{w_1}-e^{w_2}\vert\leq e^{\vert \overline{w_1}\vert+\vert \overline{w_2}\vert}\cdot \int_0^te^{[sw_1+(1-2)w_2]}ds \cdot \vert w_1-w_2\vert.  
 \label{exp1}
\end{equation} 
Letting $w\in V\setminus\R$, on the other hand, we use   
 \bgee
 [w]\leq \frac{4 \pi [w]^2}{||\nabla [w]||_2^2}+\frac{1}{\pi} ||\nabla [w]||_2^2
 \egee
 to deduce 
\begin{equation} \label{ti}
 \int_{\Om} e^{[w]}\,dx\leq C \cdot \exp \ (\frac{1}{\pi} ||\nabla [w]||_2^2), \quad w\in V,
\end{equation} 
by (\ref{ftm}). 

Inequalities (\ref{exp1})-(\ref{ti}) imply  
  \begin{eqnarray*} 
  \left|\int_{\Om} z \left(e^{w_1}-e^{w_2}\right)\,dx\right|&\leq& ||z||_4 \exp\left(\vert\bar{w}_1\vert+\vert\bar{w}_2\vert\right)\left |\left| \int_0^1 e^{[sw_1+(1-s)w_2]}\,ds\right|\right|_4 ||w_1-w_2||_2\\
  &\leq& C(K) ||z||_V ||w_1-w_2||_V, \quad \Vert w_1\Vert_V, \ \Vert w_2\Vert_V\leq K, 
  \end{eqnarray*} 
and hence (\ref{418}) is valid  due to (\ref{213}). 
\end{proof}

  \begin{lem}\label{lem2}
Given $w^*\in V$  with  $\delta\mathcal{E}(w^*)=0$, any $K>0$ admits $C=C(K)>0$ such that 
\[ 
w\in V, \ ||w||_V\leq K \quad \Rightarrow \quad 
|\mathcal{E}(w)-\mathcal{E}(w^*)|\leq C||w-w^*||_V^2.
\] 
  \end{lem}
\begin{proof}
  Since 
  \bgee
  \mathcal{E}(w)-\mathcal{E}(w^*)&=&\int_0^1 \frac{d}{ds} \mathcal{E} (sw+(1-s)w^*)\,ds\\
  &=& \int_0^1 \langle w-w^*,\delta\mathcal{E}(sw+(1-s)w^*) \rangle_{V,V^*}\,ds\\
  &=& \int_0^1 \left<w-w^*,\delta\mathcal{E}(sw+(1-s)w^*)- \delta\mathcal{E}(w^*)\right>_{V,V^*}\,ds
  \egee
we obtain 
  \bgee
   |\mathcal{E}(w)-\mathcal{E}(w^*)|&\leq&||w-w^*||_{V}\int_0^1 \left|\left|\delta\mathcal{E}(sw+(1-s)w^*)- \delta\mathcal{E}(w^*) \right|\right|_{V^*}\,ds\\
   &\leq& C(K) ||w-w^*||_{V} ||w-w^*||_V \int_0^1 s\,ds=\frac{C(K)}{2}||w-w^*||_V^2
  \egee
by the previous lemma. 
\end{proof}

We are ready to prove the key reasult in the current section.

\begin{proof}[Proof of Lemma \ref{lem51}] 

We take $w$ as in (\ref{55}). Recall $\delta\mathcal{E}(w^\ast)=0$, and deduce from (\ref{fv}) that 
\begin{eqnarray} 
-\delta\mathcal{E}(w) & = & -\delta\mathcal{E}(w)+\delta\mathcal{E}(w^*) \nonumber\\ 
& = & \Delta(w-w^*)+(e^w-e^{w^*}) \nonumber\\
&=& \Delta(w-w^*)+\int_0^1 \frac{d}{ds} e^{sw+(1-s)w^*}\,ds \nonumber\\
& = & \Delta (w-w^\ast)+\int_0^1e^{sw+(1-s)w^\ast}(w-w^\ast) \ ds \nonumber\\ 
&=&\Delta(w-w^*)+e^{w^*}(w-w^*)+\int_0^1 (e^{sw+(1-s)w^*}-e^{w^*})(w-w^*)\,d\zeta \nonumber\\ 
& = & -\mathcal{M}(w-w^\ast)+z,  
 \label{56}
\end{eqnarray} 
where 
\[ z=\int_0^1 (e^{sw+(1-s)w^*}-e^{w^*})(w-w^*)\,ds. \] 
Here we use 
\begin{eqnarray*} 
e^{sw+(1-s)w^*}-e^{w^*}&=&\int_0^1 \frac{d}{d \zeta}e^{\zeta(sw+(1-s)w^*)+(1-\zeta)w^*}\,d\zeta\\
&=& \int_0^1 e^{\zeta(sw+(1-s)w^*)+(1-\zeta)w^*}\, s(w-w^*) \ d\zeta, 
\end{eqnarray*} 
to derive 
\[ 
\vert z \vert \leq \vert w-w^*\vert ^2 e^{\vert w\vert+\vert w^*\vert}. 
\] 
Hence it holds that 
\[ 
||z||_2\leq ||\exp(\vert w\vert)||_4 \cdot ||\exp(\vert w^*\vert)||_4\cdot ||w-w^*||^2, 
\] 
and therefore, the assumption (\ref{55}) ensures 
\begin{equation}\label{est1}
\Vert z\Vert_{V^*}\leq C_1\Vert z\Vert_2\leq C_2||w-w^*||_2^2\leq C_3||w-w^*||_V^2
\end{equation} 
by the Trudinger-Moser-Fontana inequality. 

Since $\mathcal{M}: V\to V^\ast$ is provided with (\ref{psk2}), it follows that 
\[ 
\Vert w-w^*\Vert_V\leq C_1 \Vert \mathcal{M}(w-w^\ast)\Vert_{V^\ast} 
\] 
from (\ref{55}) if $\varepsilon_1\leq \varepsilon_0$. By (\ref{56})-(\ref{est1}), therefore, we obtain 
\[ \Vert w-w^\ast\Vert_V\leq C_2(\Vert \delta \mathcal{E}(w)\Vert_{V^\ast}+\Vert w-w^\ast\Vert_V^2). \] 
Choosing $0<\varepsilon_1\ll 1$ in (\ref{55}), then we reach to
\begin{equation}
\Vert w-w^\ast\Vert_{V^\ast}\leq C_3\Vert \delta \mathcal{E}(w)\Vert_{V^\ast}. 
 \label{59}
\end{equation} 

Lemma \ref{lem2} now guarantees 
\[ \vert \mathcal{E}(w)-\mathcal{E}(w^\ast)\vert \leq C_4\Vert w-w^\ast\Vert_V^2\leq C_5\Vert \delta \mathcal{E}(w)\Vert_{V^\ast}^2, \] 
which entails the conclusion of (\ref{theta}) for $\theta=1/2$ under the presense of (\ref{55}). 
\end{proof}

\section*{Acknowledgement} 

The second author was supported by Kakenhi 19H01799.  The first author would like to thank Professor Nicholas Alikakos for suggesting him the  {\L}ojasiewicz-Simon gradient inequlaity approach for the investigation of the convergence of the normalized Ricci flow.

\end{document}